\documentclass[12pt,reqno]{article}

\usepackage[usenames]{color}
\usepackage[all]{xy}
\usepackage[colorlinks=true,
linkcolor=webgreen,
filecolor=webbrown,
citecolor=webgreen]{hyperref}

\definecolor{webgreen}{rgb}{0,.5,0}
\definecolor{webbrown}{rgb}{.6,0,0}

\usepackage{amsfonts,amsthm,amssymb,amsmath}

\newcounter{makeconstant}\newenvironment{makeconstant}
{\refstepcounter{makeconstant}}{}\def\mc#1{\begin{makeconstant}\label{#1}\end{makeconstant}}

\renewcommand{\le}{\leqslant}
\renewcommand{\ge}{\geqslant}

\renewcommand{\emptyset}{\varnothing}
\renewcommand{\epsilon}{\varepsilon}
\renewcommand{\subset}{\subseteq}

\newcommand{\seqnum}[1]{\href{http://www.research.att.com/cgi-bin/access.cgi/as/~njas/sequences/eisA.cgi?Anum=#1}{\underline{#1}}}

\newcommand{\divides}{|}

\newcommand{\notdivides}{\nmid}

\def\bigo{\operatorname{O}}    %% uniform notation for 'big O'
\def\littleo{{\operatorname{o}}} %% uniform notation for 'little o'
\def\({\left(}
\def\){\right)}
\def\eul{{\rm{e}}}

\newcommand{\dirichlet}{\mathsf d}
\newcommand{\lcm}{\operatorname{lcm}}

\newcommand{\fix}{\operatorname{\mathcal F}}
\newcommand{\mertens}{\operatorname{\mathcal M}}

\newcommand{\orbit}{\operatorname{\mathcal O}}
\newcommand{\monoid}{\operatorname{\mathcal G}}

\begin{document}

%\begin{center}
%\epsfxsize=4in
%\leavevmode\epsffile{logo129.eps}
%%%%%TBC: need the eps file
%\end{center}

\begin{center}
\vskip 1cm{\LARGE\rm
Functorial orbit counting}

\vskip 1cm
\large
Apisit Pakapongpun and Thomas Ward\\
School of Mathematics\\
University of East Anglia\\
Norwich NR6 5LB, UK\\
%\href{mailto:t.ward@uea.ac.uk}\\
\end{center}

\vskip .2in

\begin{abstract}
We study the functorial and growth properties of closed orbits
for maps. By viewing an arbitrary sequence as the
orbit-counting function for a map, iterates and Cartesian
products of maps define new transformations between integer
sequences. An orbit monoid is associated to any integer
sequence, giving a dynamical interpretation of the Euler
transform.
\end{abstract}

\newtheorem{theorem}{Theorem}[section]
\newtheorem{proposition}{Proposition}[section]
\newtheorem{corollary}{Corollary}[section]
\newtheorem{lemma}{Lemma}[section]
\theoremstyle{definition}
\newtheorem{definition}{Definition}[section]
\newtheorem{example}{Example}[section]

\section{Introduction}

Many combinatorial or dynamical questions involve counting the number
of closed orbits or the periodic points under iteration of a map.
Here we consider functorial
properties of orbit-counting in the following sense. Associated
to a map~$T:X\to X$ with the property that~$T$ has only finitely many
orbits of each length are combinatorial data
(counts of fixed points and periodic orbits),
analytic data (a zeta function and a Dirichlet series) and
algebraic data (the orbit monoid). On the other hand, the collection
of such maps is closed under disjoint unions, direct products,
iteration, and other
operations. Our starting point is to ask how the associated data
behaves under those operations. A feature of this work is that
these natural operations applied to maps with simple orbit
structures give novel constructions of sequences with
combinatorial or arithmetic interest. Routine calculations are
suppressed here for brevity; complete details, related results,
and further applications will appear
in the thesis of the first author~\cite{apisit}.

We define the following categories:
{\sc maps} $\mathfrak{M}$, comprising all pairs~$(X,T)$ where~$T$
is a map~$X\to X$ with~$\fix_T(n)=\vert\{x\in X\mid
T^n(x)=x\}\vert<\infty$ for all~$n\ge1$;
{\sc orbits} $\mathfrak{O}=
\mathbb N_0^{\mathbb N}$, comprising all
sequences~$(a_n)_{n\ge1}$ with~$a_n\ge0$ for all~$n\ge1$; and
{\sc fixed points} $\mathfrak{F}\subset\mathfrak O$,
comprising any sequence~$a=(a_n)$
with the property that there is some~$(X,T)\in\mathfrak{M}$
with~$a_n=\fix_T(n)$ for all~$n\ge 1$.
For~$(X,T)\in\mathfrak M$, a closed orbit of length~$n$ under~$T$
is any set
of the form~$\tau=\{x,Tx,\dots,T^nx=x\}$ with
cardinality~$\vert\tau\vert=n$, and
we write~$\orbit_T(n)$ for the number of closed orbits of length~$n$.
Clearly
\begin{equation}\label{fixintermsoforbits}
\fix_T(n)=\sum_{d\divides n}d\orbit_T(d),
\end{equation}
so
\begin{equation}\label{orbitsintermsoffix}
\orbit_T(n)=\frac{1}{n}\sum_{d\divides
n}\mu(n/d)\fix_T(d),
\end{equation}
and this defines a bijection between~$\mathfrak F$
and~$\mathfrak O$ (see Everest, van der Poorten, Puri and the
second author~\cite{MR1938222},~\cite{MR1873399} for more on
the combinatorial applications of this bijection, and for
non-trivial examples of sequences in~$\mathfrak F$; see Baake
and Neum{\"a}rker~\cite{bn} for more on spectral
properties of the operators on~$\mathfrak O$). Since the
space~$X$ will not concern us, we will fix it to be some
countable set and refer to an element of~$\mathfrak M$ as a
map~$T$.

Recall from Knopfmacher~\cite{MR545904} that an \emph{additive
arithmetic semigroup} is a free abelian monoid~$G$ equipped
with a non-empty set of generators~$P$, and a weight
function
\[
\partial:G\to\mathbb N\cup\{0\}
\]
with~$\partial(a+b)=\partial(a)+\partial(b)$ for all~$a,b\in
G$, satisfying the finiteness property
    \[
    P(n)=\vert\{p\in P\mid\partial(p)=n\}\vert<\infty
    \]
for all~$n\ge1$. Given~$T\in\mathfrak M$, we define the orbit
monoid~$\monoid_T$ associated to~$T$ to be the free abelian
monoid generated by the closed orbits of~$T$, equipped with the
weight function
\[
\partial(a_1\tau_1+\cdots+a_r\tau_r)
=
a_1\vert\tau_1\vert+\cdots+a_r\vert\tau_r\vert,
\]
and write~$\monoid_T(n)$ for the number of elements of weight~$n$.
Finally, define {\sc orbit monoids}~$\mathfrak G$ to be the category of all
such monoids associated to maps in~$\mathfrak M$. Notice that every
additive arithmetic semigroup in the sense of Knopfmacher~\cite{MR545904}
is an orbit monoid,
since for any sequence~$(a_n)$ there is a
map~$T$ with~$\orbit_T(n)=a_n$ for all~$n\ge1$ (indeed,
Windsor~\cite{MR2422026} shows that
the map~$T$ may be chosen to be a~$C^{\infty}$ diffeomorphism of a torus).
We will write~$\monoid_T$
for the sequence~$\(\monoid_T(n)\)$, since the sequence
determines the monoid up to isomorphism.

As usual, we write~$\zeta=(1,1,1,\dots)$ and~$\mu=(1,-1,-1,0,\dots)$
for the zeta and
M{\"o}bius functions viewed as sequences, and use the same symbols
to denote their
Dirichlet series.

There are natural generating functions associated to an
element~$T\in\mathfrak M$.
If~$\fix_T(n)$ is exponentially bounded then the
\emph{dynamical zeta function}
\[
\zeta_T(s)=\exp\sum_{n\ge1}\frac{s^n}{n}\fix_T(n)
\]
converges in some complex disk (see Artin and Mazur~\cite{am}).
If~$\orbit_T(n)$ is polynomially bounded then the \emph{orbit
Dirichlet series}
\[
\dirichlet_T(s)=\sum_{n\ge1}\frac{\orbit_T(n)}{n^s},
\]
converges in some half-plane. The basic
relation~\eqref{fixintermsoforbits} is expressed in terms of
these generating functions by the two identities
\[
\zeta_T(s)=\prod_{n\ge1}\(1-s\)^{-\orbit_T(n)}
=\prod_{\tau}\(1-s^{\vert\tau\vert}\)^{-1},
\]
where the product is taken over all closed orbits of~$T$, and
\[
\dirichlet_T(s)\zeta(s+1)=\sum_{n\ge1}\frac{\fix_T(n)}{n^{s+1}}.
\]

Finally, a natural measure of the rate of growth in~$\orbit_T$ is
the number
\[
\pi_T(N)=\vert\{\tau\mid\vert\tau\vert\le N\}\vert
\]
of closed orbits of length no more than~$N$. Asymptotics
for~$\pi_T$ are analogous to the prime number theorem.
In the case of exponential
growth, that is under the assumption that~$\limsup_{n\to\infty}\frac{1}{n}
\log\orbit_T(n)=h>0$, a more smoothly averaged measure
of orbit growth is given by
\[
\mertens_T(N)=\sum_{\vert\tau\vert\le N}\frac{1}{\eul^{h\vert\tau\vert}},
\]
and asymptotics for~$\mertens_T$ are analogous to Mertens' theorem.

\section{Functorial properties}

Most functorial properties are immediate, so we simply record them
here. Write~$T_1\times T_2$ for the Cartesian product of two
maps,~$T_1\sqcup T_2$ for the disjoint union,
defined by
\begin{equation*}
(T_1\sqcup T_2)(x)=\begin{cases}T_1(x)&\mbox{if }x\in X_1,\\
T_2(x)&\mbox{if }x\in X_2,
\end{cases}
\end{equation*}
and write~$T^k$ with~$k\ge1$ for the~$k$th iterate of~$T$.
Then
\begin{enumerate}
\item $\fix_{T_1\times T_2}=\fix_{T_1}\fix_{T_2}$ (pointwise product);
\item $\dirichlet_{T_1\sqcup T_2}=\dirichlet_{T_1}+\dirichlet_{T_2}$;
\item $\zeta_{T_1\sqcup T_2}=\zeta_{T_1}\zeta_{T_2}$;
\item $\fix_{T^k}(n)=\fix_{T}(kn)$ for all~$k\ge1$ and~$n\ge1$.
\end{enumerate}

In contrast to the first of these, it is clear that
computing the
number of closed orbits under the Cartesian product of two maps
is more involved.

\begin{lemma}\label{I think as I please}
$\orbit_{T_1\times T_2}(n)=
\displaystyle\sum_{\genfrac{}{}{0pt}{}{d_1,d_2\in\mathbb N,}{\lcm(d_1,d_2)=n}}
\orbit_{T_1}(d_1)\orbit_{T_2}(d_2)\gcd(d_1,d_2).$
\end{lemma}

\begin{proof}
If~$(x_1,x_2)$ lies on a~$T_1\times T_2$-orbit of length~$n$,
then, in particular,
\[
T_1^n(x_1)=x_1
\]
and
\[
T_2^n(x_2)=x_2,
\]
so~$x_i$ lies on a~$T_i$-orbit of length~$d_i$ for some~$d_i$
dividing~$n$, for~$i=1,2$. On the other hand, if~$x_i$ lies on
a~$T_i$-orbit of length~$d_i$ for~$i=1,2$ then the~$T_1\times
T_2$-orbit of~$(x_1,x_2)$ has cardinality~$\lcm(d_1,d_2)$. On
the other hand, if~$\tau_1,\tau_2$ are orbits of
length~$d_1,d_2$ with~$\lcm(d_1,d_2)=n$, then there
are~$d_1d_2$ points in the set~$\tau_1\times\tau_2$, so this
must split up into~$d_1d_2/n=\gcd(d_1,d_2)$ orbits of
length~$n$ under~$T_1\times T_2$.
\end{proof}

\begin{example}\label{exampleof11cross11andcyclicgroups}
Let~$T$ be a map with one orbit of each length, so~$\dirichlet_T(s)=\zeta(s)$.
Then, by Lemma~\ref{I think as I please} and a calculation,
\begin{equation*}
\orbit_{T\times T}(n)=\sum_{\genfrac{}{}{0pt}{}{d_1,d_2\in\mathbb N,}{\lcm(d_1,d_2)=n}}
\gcd(d_1,d_2)
=\sum_{d\divides n}\sigma(d)\mu(n/d)^2,
\end{equation*}
so
\begin{equation}\label{equation:ttimestseries}
\dirichlet_{T\times T}(s)=\frac{\zeta(s)^2\zeta(s-1)}{\zeta(2s)}=
1+\frac{4}{2^s}+\frac{5}{3^s}+\frac{10}{4^s}+\frac{7}{5^s}+
\frac{20}{6^s}+\frac{9}{7^s}+\frac{22}{8^s}+\cdots.
\end{equation}
By identifying an orbit of length~$n$ under~$T$ with a cyclic
group~$C_n$ of order~$n$, we see from the proof of Lemma~\ref{I
think as I please} that~$\orbit_{T\times T}(n)$ is the number
of cyclic subgroups of~$C_n\times C_n$, so~$\orbit_{T\times T}$ is
\seqnum{A060648}.
\end{example}

Example~\ref{exampleof11cross11andcyclicgroups} is generalized
in Example~\ref{examplewithgeneralP}, where it corresponds to the
case~$P=\emptyset$.

\begin{example}
Let~$p$ be a prime, and assume
that~$\orbit_{T}(n)=p^n$ for all~$n\ge 1$, so~$\zeta_T(s)=\frac{1}{1-ps}$.
Then
\[
\orbit_{T\times T}(n)=
\frac{1}{n}\sum_{d\divides n}\(\mu(n/d)\sum_{d_1\divides d}d_1p^{d_1}
\sum_{d_2\divides d}d_2p^{d_2}\)=\frac{1}{n}\sum_{d\divides n}\mu(d)p^{2n/d},
\]
which is the number of irreducible polynomials of degree~$n+1$ over~$\mathbb F_{p^2}$.
In the case~$p=2$ this gives the sequence~\seqnum{A027377}, and in the
case~$p=3$ this gives~\seqnum{A027381}.
\end{example}

The behavior of orbits under iteration is more involved.
To motivate the rather dense formula below, consider the orbits
of length~$n$ under~$T^p$ for some prime~$p$. Points on an~$m$-orbit
under~$T$ lie on an orbit of length~$m/\gcd(m,p)$ under~$T^p$.
If~$m\neq n,np$ then~$m/\gcd(m,p)\neq n$, so the only points that
can contribute to~$\orbit_{T^p}(n)$ are points lying on~$n$-orbits
or on~$np$-orbits under~$T$. Each~$np$-orbit under~$T$ splits
into~$p$ orbits of length~$n$ under~$T^p$. An~$n$-orbit under~$T$ defines
an~$n$-orbit under~$T^p$ only if~$p\notdivides n$. It follows that
\begin{equation}\label{equationpowersprimeonly}
\orbit_{T^p}(n)=
\begin{cases}p\orbit_{T}(pn)+\orbit_T(n)&\mbox{if }p\notdivides n;\\
p\orbit_{T}(pn)&\mbox{if }p\divides n.
\end{cases}
\end{equation}
In order to state the general case, fix the power~$m$ and
write~$m=\boldsymbol p^{\boldsymbol a}=p_1^{a_1}\cdots p_r^{a_r}$ for the
decomposition into primes of~$m\in\mathbb N$;
for any set~$J\subset I=\{p_1,\dots,p_r\}$
write~$\boldsymbol p_J^{\boldsymbol a_J}=\prod_{p_j\in J}p_j^{a_j}$.
Finally, write~$\mathcal{D}(n)$ for the set of prime divisors of~$n$.

\begin{theorem}\label{theorem:orbitsforiteration}
Let~$m=\boldsymbol p^{\boldsymbol a}$
and~$J=J(n)=\mathcal{D}(m)\setminus\mathcal{D}(n)$. Then
\begin{equation}\label{equationpowersA}
\orbit_{T^m}(n)=\sum_{d\divides\boldsymbol p_J^{\boldsymbol a_J}}
\textstyle\frac{m}{d}\orbit_{T}(\frac{mn}{d}).
\end{equation}
\end{theorem}

\begin{proof}
Notice that~$J$ depends on~$n$, so the formula~\eqref{equationpowersA}
involves a splitting into cases depending on the primes dividing~$n$, just as
in the case of a single prime discussed above.
We argue by induction on the length~$\sum_{i=1}^{r}a_i$ of~$m$.
If the length of~$m$ is~$1$, then~$m$ is a prime and~\eqref{equationpowersA}
reduces to~\eqref{equationpowersprimeonly}. Assume now that~\eqref{equationpowersA}
holds for~$\sum_{i=1}^{r}a_i\le k$, and let~$m$ have length~$k$;
write~$I=\mathcal{D}(m)$ and~$J=\mathcal{D}(m)\setminus\mathcal{D}(n)$.
We consider the effect of multiplying~$m$ by a prime~$q$ on the
formula~\eqref{equationpowersA},
and write~$I'=\mathcal{D}(mq)$,~$J'=\mathcal{D}(mq)\setminus\mathcal{D}(n)$.

If~$q\in\mathcal{D}(n)\setminus I$ then by~\eqref{equationpowersprimeonly} we have
\begin{equation*}
\orbit_{T^{mq}}(n)=q\orbit_{T^m}(qn)=
q\sum_{d\divides\boldsymbol p_J^{\boldsymbol a_J}}
\textstyle\frac{m}{d}\orbit_{T}(\frac{mnq}{d})=
q\displaystyle\sum_{d\divides\boldsymbol p_{J'}^{\boldsymbol a_{J'}}}
\textstyle\frac{m}{d}\orbit_T(\frac{mnq}{d}),
\end{equation*}
in accordance with~\eqref{equationpowersA}.

If~$q\in I\cap\mathcal{D}(n)$, then~$I'=I$ and~$J'=J$, so
\begin{equation*}
\orbit_{T^{mq}}(n)=q\orbit_{T^m}(qn)=
q\sum_{d\divides\boldsymbol p_{J'}^{\boldsymbol a_{J'}}}\textstyle\frac{m}{d}\orbit_{T}(\frac{mnq}{d})
\end{equation*}
as required.

If~$q\notin I\cup\mathcal{D}(n)$ then
\begin{eqnarray*}
\orbit_{T^{mq}}(n)&=&q\orbit_{T^{m}}(qn)+\orbit_{T^m}(n)\\&=&
q\sum_{d\divides\boldsymbol p_J^{\boldsymbol a_J}}
\textstyle\frac{m}{d}\orbit_{T}(\frac{mqn}{d})+
\displaystyle\sum_{d\divides\boldsymbol p_J^{\boldsymbol a_J}}\textstyle\frac{m}{d}\orbit_{T}(\frac{mn}{d})\\
&=&
\displaystyle\sum_{d\divides\boldsymbol p_J^{\boldsymbol a_J}q}\textstyle\frac{qm}{d}\orbit_{T}(\frac{mqn}{d})
\end{eqnarray*}
as required.

Finally, if~$q\in I\setminus\mathcal{D}(n)$ then
\begin{eqnarray*}
\orbit_{T^{mq}}(n)&=&q\orbit_{T^{m}}(qn)+\orbit_{T^m}(n)\\
&=&
q\displaystyle
\sum_{d\divides\boldsymbol p_{J\setminus\{q\}}^{\boldsymbol a_{J\setminus\{q\}}}}
\textstyle\frac{m}{d}\orbit_{T}(\textstyle\frac{mqn}{d})+\displaystyle
\sum_{d\divides\boldsymbol p_J^{\boldsymbol a_J}}
\textstyle\frac{m}{d}\orbit_{T}(\frac{mn}{d})\\
&=&
\displaystyle\sum_{\genfrac{}{}{0pt}{}{d\divides\boldsymbol p_J^{\boldsymbol a_J}q,}{q\notdivides d}}
\textstyle\frac{mq}{d}\orbit_{T}\textstyle\frac{mqn}{d}+
\displaystyle\sum_{\genfrac{}{}{0pt}{}{d\divides\boldsymbol p_J^{\boldsymbol a_J}q,}{q\divides d}}
\textstyle\frac{m}{d/q}\orbit_{T}(\textstyle\frac{mn/d}{d/q})\\
&=&
\displaystyle\sum_{d\divides\boldsymbol p_J^{\boldsymbol a_J}q}\frac{qm}{d}\orbit_{T}(\textstyle\frac{mqn}{d}),
\end{eqnarray*}
completing the proof.
\end{proof}

This defines a family of transformations on sequences,
taking~$\orbit_T$ to~$\orbit_{T^k}$ for each~$k\ge1$.

\begin{example}
Let~$T\in\mathfrak M$ have~$\orbit_T(n)=n$ for all~$n\ge 1$, so~$\dirichlet_T(s)=\zeta(s-1)$.
Then by~\eqref{orbitsintermsoffix} we have
\[
\orbit_{T^2}(n)=\frac{1}{n}\sum_{d\divides n}\mu\(\frac{n}{d}\)\fix_{T^2}(d)=
\frac{1}{n}\sum_{d\divides n}\mu\(\frac{n}{d}\)\sigma_2(2d)
\]
since~$\fix_{T^2}(d)=\fix_T(2d)=\sum_{e\divides
2d}e\orbit_T(e)=\sum_{e\divides 2d}e^2$, so
\[
\dirichlet_{T^2}(s)=\(5-\frac{2}{2^s}\)\zeta(s-1)=
5+\frac{8}{2^s}+\frac{15}{3^s}+\frac{16}{4^s}+\frac{25}{5^s}+
\frac{24}{6^s}+\frac{35}{7^s}+\frac{32}{8^s}+\cdots.
\]
Thus~$\orbit_{T^2}$ is~\seqnum{A091574} (up to an offset).
\end{example}

\begin{example}
More generally, if~$\dirichlet_T(s)=\zeta(s-1)$ and~$p$ is a prime, then
\[
\orbit_{T^p}(n)=\frac{1}{n}\sum_{d\divides n}\mu\(\frac{n}{d}\)\sigma_2(pd)=
\begin{cases}(p^2+1)n&\mbox{if }p\notdivides n;\\
p^2n&\mbox{if }p\divides n\end{cases}
\]
so~$\dirichlet_{T^p}(s)=\(p^2+1-\displaystyle\frac{p}{p^s}\)\zeta(s-1)$.
Composite powers are more involved; full details are
in~\cite{apisit}. For example,
\[
\orbit_{T^4}(n)=\begin{cases}16n&\mbox{if $n$ is even};\\
21n&\mbox{if $n$ is odd}
\end{cases}
\]
so~$\dirichlet_{T^4}(s)=\(2-\frac{10}{2^s}\)\zeta(s-1).$
\end{example}

An important family of dynamical systems -- those of
\emph{finite combinatorial rank} -- have
been studied by Everest, Miles, Stevens and Ward~\cite{emsw}. These have the
property that their orbit Dirichlet series is ``Dirichlet--rational'',
that is there is
a finite set~$C\subset\mathbb Z$ with the
property that~$\dirichlet_T(s)$ is a rational
function in the variables~$\{c^{-s}\mid c\in C\}$.
An easy consequence of Theorem~\ref{theorem:orbitsforiteration} is that
this property is preserved under iteration.

\begin{corollary}
If maps~$S$ and~$T$ have Dirichlet--rational orbit Dirichlet
series, then so do~$S\times T$ and~$T^k$ for any~$k\ge1$.
\end{corollary}

\begin{example}\label{feigenbaumquadraticmap}
The quadratic map~$T:x\mapsto1-cx^2$ on the
interval~$[-1,1]$ at the Feigenbaum
value~$c=1.401155\cdots$ (see
Feigenbaum's lecture
notes~\cite{MR1209847}; this is at the end of a period-doubling
cascade) gives a particularly
simple example of a Dirichlet--rational Dirichlet series.
This map has
\[
\orbit_{T}(n)=\begin{cases}1&\mbox{if }n=2^k
\mbox{ for some }k\ge0;\\
0&\mbox{if not},
\end{cases}
\]
so~$\dirichlet_T(s) =\frac{1}{1-2^{-s}}$
and~$\orbit_T$ is (up to an offset) the Fredholm-Rueppel sequence~\seqnum{A036987}.
By~\eqref{fixintermsoforbits} we have~$\fix_T(n)=2\lfloor n
\rfloor_2-1$, where~$\lfloor n\rfloor_2=\vert n\vert_2^{-1}$ denotes the~$2$-part of~$n$, so~$\fix_T$
is~\seqnum{A038712}. Using this we see that
\[
\orbit_{T^2}(n)=\frac{1}{n}\sum_{d\divides n}\mu\(\frac{n}{d}\)
\(2\lfloor 2d\rfloor_2-1\)=
\begin{cases}
3&\mbox{if }n=1;\\
2&\mbox{if }n=2^k
\mbox{ for some }k\ge1;\\
0&\mbox{if not},
\end{cases}
\]
so
\[
\dirichlet_{T^2}(s)=\frac{3-2^{-s}}{1-2^{-s}}.
\]
More generally, the formula for~$\fix_T$ shows that
\[
\dirichlet_{T^k}(s)=\lfloor k\rfloor_2-1+\lfloor k\rfloor_2\dirichlet_{T}(s)
\]
for any~$k\ge1$.
\end{example}

\begin{example}
With~$T$ as in Example~\ref{feigenbaumquadraticmap}, a similar
calculation using Lemma~\ref{I think as I please} shows that
\[
\dirichlet_{T\times
T}(s)=\frac{3}{1-2^{-(s-1)}}-\frac{2}{1-2^{-s}}.
\]
If~$S\in\mathfrak M$ has
\[
\orbit_{S}(n)=\begin{cases}1&\mbox{if }n=3^k
\mbox{ for some }k\ge0;\\
0&\mbox{if not},
\end{cases}
\]
then~$O_{T\times S}$ is \seqnum{A065333}, the characteristic
function of the~$3$-smooth numbers, so
\[
\dirichlet_{T\times S}(s)=\frac{1}{(1-2^{-s})(1-3^{-s})},
\]
\[
\dirichlet_{T\times T\times S}(s)
=
\frac{3}{(1-2^{-(s-1)})(1-3^{-s})}-\frac{2}{(1-2^{-s})(1-3^{-s})},
\]
and so on.
\end{example}

The thesis of the first author~\cite{apisit} characterizes the
existence of ``roots'': that is, given a
sequence~$a\in\mathfrak O$ and~$k\ge1$ to determine if there is
some~$T\in\mathfrak M$ with~$\orbit_{T^k}=a$. Instances of no
roots, unique roots, and uncountably many roots occur.

\section{Multiplicative sequences}

Multiplicative sequences in~$\mathfrak O$ are particularly easy
to work with, and in this section we describe simple examples
of such sequences, and some properties of their product
systems. In particular, we show how simple orbit sequences may
factorize (that is, be the orbit sequence of the product of two
maps) in many different ways. Since~$\fix_T$ and~$\orbit_T$ are
related by convolution with~$\mu$ and multiplication by~$n$, it
is clear that~$\fix_T$ is multiplicative if and only
if~$\orbit_T$ is multiplicative. The next lemma is equally
straightforward; we include the proof to illustrate how the
correspondence between~$\fix_T$ and~$\orbit_T$ may be
exploited.

\begin{lemma}
If any two of~$\orbit_T$,~$\orbit_{S}$ and~$\orbit_{T\times S}$ are
multiplicative, then so is the third.
\end{lemma}

\begin{proof}
Assume the first two are multiplicative and~$\gcd(m,n)=1$. Then
\begin{eqnarray*}
\orbit_{T\times S}(mn)&=&\textstyle\frac{1}{mn}
\displaystyle\sum_{d\divides mn}\mu(\textstyle\frac{mn}{d})\fix_{T\times S}(d)
=\textstyle\frac{1}{mn}\displaystyle\sum_{d\divides m}
\sum_{d'\divides n}\mu(\textstyle\frac{m}{d})\mu(\frac{n}{d'})\fix_{T}(dd')\fix_S(dd')\\
&=&
\textstyle\frac{1}{mn}\displaystyle\sum_{d\divides m}\sum_{d'\divides n}
\mu(\textstyle\frac{m}{d})\mu(\frac{n}{d'})
\fix_{T}(d)\fix_{T}(d')\fix_{S}(d)\fix_{S'}(d')\\
&=&
\textstyle\frac{1}{mn}\displaystyle\sum_{d\divides m}\sum_{d'\divides n}
\mu(\textstyle\frac{m}{d})\mu(\frac{n}{d'})
\fix_{T\times S}(d)\fix_{T\times S}(d')=\orbit_{T\times S}(m)\orbit_{T\times S}(n).
\end{eqnarray*}

Now assume that~$\orbit_S$ is not multiplicative
while~$\orbit_T$ is, and choose~$m,n$ of minimal product
with the property
that~$\gcd(m,n)=1$ and~$\orbit_{S}(mn)\neq\orbit_S(m)\orbit_S(n).$
Then, by construction, if~$ab<mn$ and~$\gcd(a,b)=1$ we
have~$\orbit_{S}(ab)=\orbit_S(a)\orbit_S(b)$, so we must
have~$\fix_S(mn)\neq\fix_S(m)\fix_S(n)$.
If~$mn=1$ then
\[
\orbit_{T\times S}(1)=\fix_{T\times S}(1)=
\fix_{T}(1)\fix_{S}(1)=\orbit_T(1)\orbit_S(1)=\orbit_S(1)\neq1,
\]
so~$\orbit_{T\times S}$ is not multiplicative.
If~$mn>1$ then a calculation gives
\begin{eqnarray*}
\orbit_{T\times S}(mn)
&=&
{\textstyle\frac{1}{mn}}
\sum_{\genfrac{}{}{0pt}{}{d\divides m,d'\divides n,}{dd'<mn}}
\mu({\textstyle\frac{mn}{dd'}})\fix_{T\times S}(dd')
+
{\textstyle\frac{1}{mn}}\fix_{T\times S}(mn)\\
&=&{\textstyle\frac{1}{m}}
\sum_{d\divides m}\mu({\textstyle\frac{m}{d}})\fix_{T\times S}(d)\cdot
{\textstyle\frac{1}{n}}
\sum_{d'\divides n}\mu({\textstyle\frac{n}{d'}})\fix_{T\times S}(d')\\
&&\qquad\qquad-
\textstyle\frac{1}{mn}
\fix_{T\times S}(m)
\fix_{T\times S}(n)+\frac{1}{mn}\fix_{T\times S}(mn)\\
&=&\orbit_{T\times S}(m)
\orbit_{T\times S}(n)-\textstyle\frac{1}{mn}\fix_{T}(mn)
\underbrace{\fix_S(m)\fix_S(n)}_{\neq\fix_S(mn)}\\
&\neq&\orbit_{T\times S}(m)\orbit_{T\times S}(n)
\end{eqnarray*}
since~$\fix_{T}(mn)\ge\fix_T(1)=1$,
so~$\orbit_{T\times S}$ is not multiplicative.
\end{proof}

This gives a bijective proof of~\eqref{equation:ttimestseries} as follows.
Write~$\ell(n)$ for the number of primitive lattices of
index~$n$ in~$\mathbb Z^2$, so that~\eqref{equation:ttimestseries}
is equivalent to the statement
\[
\orbit_{T\times T}(n)=\sum_{d\divides n}\ell(d).
\]
Both sides of this equation are multiplicative,
so it is enough to prove this for~$n=p^r$ a prime power.
The primitive lattices of index~$p^j$ in~$\mathbb Z^2$ are
in one-to-one correspondence with
\[
\left\{\begin{bmatrix}a&b\\0&c\end{bmatrix}\mid
ac=p^j,~a,c\ge1,~0\le b<p^j,~\gcd(a,b,c)=1\right\}.
\]
It follows that~$\sum_{d\divides p^r}\ell(d)=p^r+2\sum_{j=0}^{r-1}p^j$,
in agreement with the formula for~$\orbit_{T\times T}(p^r)$.

Write~$\mathbb P$ for the set of all prime numbers,
and for a subset~$P\subset\mathbb P$ write~$P^c=\mathbb P\setminus P$.

\begin{example}\label{exampleofSPsequence}
For any set~$P$ of primes, define~$s_P\in\mathfrak O$ by
\[
s_P(n)=\begin{cases}0&\mbox{if }p\divides n\mbox{ for some }p\in P;\\
1&\mbox{if not.}\end{cases}
\]
\end{example}

\begin{lemma}\label{extralemma}
If~$T$ is a map with~$\orbit_T=s_P$ and~$k\ge1$, then
\[
\orbit_{n}(T^{k})=\left\{\begin{array}{ll}
\displaystyle\prod_{p\in{Q}, p\divides n}
p^{a_p}\cdot
\displaystyle\prod_{p\in{Q}, p\notdivides n}
\sigma({p^{a_p}})
&\mbox{if }p\notdivides{n}\mbox{ for all }p\in {P}; \\
0&\mbox{if }p\divides n\mbox{ for some }p\in{P},\end{array}\right.
\]
where~$k=\displaystyle\prod_{p\in\mathcal{P}}p^{a_p}\cdot
\displaystyle\prod_{p\in\mathcal{Q}}p^{a_p}$
with~$P\cap Q=\emptyset$ is the prime decomposition of~$k$.
\end{lemma}

\begin{proof}
Let~$J\subset P\cup Q$ and~$I\subset Q$.
Then
\begin{eqnarray*}
\orbit_{T^k}(n)
&=&
\sum_{d\divides{\boldsymbol p_J}^{\boldsymbol a_J}}(k/d)
\orbit_{T}(kn/d)\\
&&\qquad\qquad\qquad(\mbox{where }p\notdivides n
\mbox{ for }p\in J,
p\divides n\mbox{ for }p\in(P\cup Q)\setminus J)\\
&=&\sum_{d\divides\prod_{p\notdivides n}p^{a_p}}
\prod_{p\in Q}({p^{a_p}}/d)
\orbit_T(p^{a_p}n/d),
\end{eqnarray*}
showing the second case.
If~$p\notdivides n$ for~$p\in I$,~$p\divides n$
for~$p\in Q\setminus I$ and~$p\notdivides n$
for any~$p\in P$
then~$\orbit_T(p^{a_p}n/d)=1$, so
\begin{eqnarray*}
\orbit_{T^k}(n)&=&\sum_{d\divides\prod_{p\notdivides n}
p^{a_p}}\prod_{p\in Q}(p^{a_{p}}/d)\\
&=&
\left(\prod_{p\in Q}p^{a_p}\right)
\sum_{d\divides\prod_{p\notdivides n}p^{a_p}}1/d\\
&=&
\prod_{p\in Q}p^{a_p}
\left(
\sum_{d\divides\prod_{p\notdivides n}p^{a_p}}
d/{\textstyle\prod}_{p\notdivides n}p^{a_p}\right)\\
&=&
\sum_{d\divides\prod_{p\notdivides n}p^{a_p}}
\left(d/{\textstyle\prod}_{p\notdivides n}p^{a_p}
\right)\prod_{p\notdivides n}
p^{a_p}\prod_{p\divides n}p^{a_p}\\
&=&
\prod_{p\divides n}p^{a_p}
\left(\prod_{p\notdivides n}
\sum_{d\divides p^{a_p}}d\right)\\
&=&
\prod_{\substack{p\in Q\\p\divides n}}
p^{a_p}
\cdot
\prod_{\substack{p\in Q\\p\notdivides n}}
\sigma(p^{a_p}),
\end{eqnarray*}
showing the first case.
\end{proof}

It is clear from Lemma~\ref{I think as I please}
that if~$\orbit_S=s_P$ and~$\orbit_T=s_{P^c}$
then~$\orbit_{T\times S}=\zeta$,
so the sequence~$\zeta$ factorizes in uncountably
many ways into the orbit count of two combinatorially
distinct systems. Indeed, these sequences provide the \emph{only}
combinatorial factorization of~$\zeta$ into the orbit
count of the product of two systems.

\begin{proposition}
If~$S$ and~$T$ are maps with~$\orbit_{S\times T}=\zeta$,
then there is a set~$P\subset\mathbb P$
for which~$\orbit_T=s_{p}$ and~$\orbit_S=s_{P^c}$.
\end{proposition}

\begin{proof}
It is clear from Lemma~\ref{I think as I please}
that~$\orbit_S$ and~$\orbit_T$ take values in~$\{0,1\}$,
and moreover that~$\{\orbit_S(p),\orbit_T(p)\}=\{0,1\}$ for~$p\in\mathbb P$.
Fix a pair of maps satisfying the hypothesis,
and let~$P=\{p\in\mathbb P\mid\orbit_T(p)=0\}$,
so that~$P^c=\{p\in\mathbb P\mid\orbit_S(p)=0\}$.

Assume that~$p\divides n$ for some~$p\in P$,
so that~$\orbit_T(p)=0$ and~$\orbit_S(p)=1$.
If~$\orbit_T(n)=1$, then
\[
1=\sum_{\lcm(d,d')=n}\gcd(d,d')\orbit_T(d)\orbit_S(d')\ge p,
\]
which is impossible, so~$\orbit_T(n)=0$.
By symmetry, if~$p\divides n$ for some~$p\in P^c$,
then~$\orbit_S(n)=0$.

Now if~$n\neq1$ is not divisible by any~$p\in P$, then
\begin{eqnarray*}
1&=&\sum_{\lcm(d,d')=n}\gcd(d,d')\orbit_T(d)\orbit_S(d')\\
&=&\sum_{\genfrac{}{}{0pt}{}{\lcm(d,d')=n,}{\mathcal D(d)\subset P^c,\mathcal D(d')\subset P}}
\gcd(d,d')\orbit_T(d)\orbit_S(d')\\
&=&\orbit_T(n)
\sum_{d'\divides n}d'\orbit_S(d'),
\end{eqnarray*}
so~$\orbit_T(n)=1$.
It follows that~$\orbit_T=s_P$, and by symmetry~$\orbit_S=s_{P^c}$ as required.
\end{proof}

Products of the systems in Example~\ref{exampleofSPsequence}
enjoy remarkable combinatorial properties, illustrated in the
examples below. The calculations in the examples all follow from
the next lemma. Let~$S\subset\mathbb P$ be a set of primes.
Write~$\lfloor n\rfloor_{S}$ for the $S$-part of~$n$, that is
\[
\lfloor n\rfloor_S=\prod_{p\in S}\vert n\vert_p^{-1}.
\]
Write~$\gcd(n,S)$ as shorthand for~$\gcd(n,\prod_{p\in S}p)$.

\begin{lemma}For any set~$S\subset\mathbb P$,
\begin{equation}\label{equation:interpolationthing}
\sum_{n\ge1}\frac{\lfloor n\rfloor_S}{n^s}
=
\zeta(s)\prod_{p\in S}\left( \frac{p^s-1}{p^s-p} \right).
\end{equation}
\end{lemma}

\begin{proof}
Recall that
\begin{equation}\label{eulerone}
\sum_{n\ge1,\gcd(n,S)=1}\frac{1}{n^s}=\prod_{p\in S}\(1-p^{-s}\)\zeta(s).
\end{equation}
Write~$S=\{p_1,\dots\}$.
Then, writing~$n=p_1^{a_1}\cdots p_r^{a_r}m$ with~$\gcd(m,S)=1$,
\begin{eqnarray*}
\sum_{n\ge1}\frac{\lfloor n\rfloor_S}{n^s}&=&
\sum_{\gcd(m,S)=1}\sum_{a_1\ge0}\cdots\sum_{a_r\ge0}
\frac{p_1^{a_1}\cdots p_r^{a_r}}{(p_1^{a_1}\cdots p_r^{a_r})^sm^s}\\
&=&\sum_{\gcd(m,S)=1} \frac{1}{m^s}\sum_{a_2\ge0}\cdots\sum_{a_{r}\ge0}
\frac{p_2^{a_2}\cdots p_{r}^{a_r}}{(p_2^{a_2}\cdots p_r^{a_r})^s}\sum_{a_1\ge0}
\frac{1}{(p_1^{a_1})^{s-1}}\\
&=&\sum_{\gcd(m,S)=1} \frac{1}{m^s}\sum_{a_2\ge0}\cdots\sum_{a_{r}\ge0}
\frac{p_2^{a_2}\cdots p_{r}^{a_r}}{(p_2^{a_2}\cdots p_r^{a_r})^s}
\(\frac{1}{1-p_1^{-(s-1)}}\),
\end{eqnarray*}
so by induction we have
\begin{eqnarray*}
\sum_{n\ge1}\frac{\lfloor n\rfloor_S}{n^s}&=&
\sum_{\gcd(m,S)=1}
\frac{1}{m^s}\prod_{p\in S}\(\frac{1}{1-p^{-(s-1)}}\)\\
&=&\zeta(s)\prod_{p\in S}\(1-p^{-s}\)\frac{1}{1-p^{-(s-1)}}
\end{eqnarray*}
since I can write~$\sum_{\gcd(m,S)=1}$ as~$(1-p_1^{-s})\sum_{\gcd(m,S\setminus\{p_1\})=1}$
as in~\eqref{eulerone},
as required.
\end{proof}

Notice that~\eqref{equation:interpolationthing} interpolates between~$\zeta(s)$
(when~$S=\emptyset$) and~$\zeta(s-1)$ (when~$S=\mathbb P$). Of course how the
abscissa of convergence moves from~$1$ at~$S=\emptyset$ to~$2$ at~$S=\mathbb P$
is rather subtle. A similar argument gives the following.

\begin{lemma}\label{lemma:technicaltrick}
Let~$S$ be a set of primes.
Then
\[
a_{S,n}=\prod_{p\in S}\(\frac{1}{p-1}\)\(\frac{p+1}{\vert n\vert_p}-2\)
\]
for all~$n\ge1$
if and only if
\[
\sum_{n\ge1}\frac{a_{S,n}}{n^s}=
\zeta(s)\prod_{p\in S}\frac{p^s+1}{p^s-p}.
\]
\end{lemma}

\begin{proof}
First, if~$q\notin S$ and~$\gcd(m,q)=1$,
then
\begin{equation}\label{eulerthree}
a_{S,mq^k}=a_{S,m}
\end{equation}
for all~$k\ge0$,
since~$\vert q\vert_p=1$ for all~$p\in S$.
Second, we have an extension of the identity~\eqref{eulerone}:
if~$q\notin S$, then
\begin{equation}\label{eulerfour}
\sum_{m\ge1,\gcd(m,q)=1}\frac{a_{S,m}}{m^s}=
\(1-q^{-s}\)\sum_{n\ge1}\frac{a_{S,n}}{n^s}
\end{equation}
by the usual argument and~\eqref{eulerthree}.
We now prove the lemma by induction on the cardinality of~$S$.
Assume we have the lemma for some set~$S$,
and assume that~$q\notin S$. Write~$S'=S\cup\{q\}$,
and notice that
\begin{eqnarray*}
\sum_{n\ge1}\frac{a_{S',n}}{n^s}&=&
\sum_{n\ge1}\frac{\frac{1}{q-1}\(\frac{q+1}{\vert n\vert_q}-2\)a_{S,n}}{n^s}\\
&=&
\(\frac{q+1}{q-1}\)\sum_{n\ge1}\frac{a_{S,n}/\vert n\vert_q}{n^s}-\frac{2}{q-1}
\sum_{n\ge1}\frac{a_{S,n}}{n^s}\\
&=&\(\frac{q+1}{q-1}\)\sum_{m\ge1,\gcd(m,q)=1}
\sum_{k\ge0}\frac{q^ka_{S,mq^k}}{(q^k)^sm^s} -\frac{2}{q-1}\sum_{n\ge1}\frac{a_{S,n}}{n^s}\\
&=&\(\frac{q+1}{q-1}\)\sum_{m\ge1,\gcd(m,q)=1}
\frac{a_{S,m}}{m^s}\sum_{k\ge0}\frac{1}{(q^k)^{s-1}}
-\frac{2}{q-1}\sum_{n\ge1}\frac{a_{S,n}}{n^s}\\
&&\qquad\mbox{by~\eqref{eulerthree}}\\
&=&\frac{q+1}{q-1}\(1-q^{-s}\)\frac{1}{1-q^{1-s}}
\sum_{n\ge1}\frac{a_{S,n}}{n^s}-\frac{2}{q-1}
\sum_{n\ge1}\frac{a_{S,n}}{n^s} \\
&&\qquad\mbox{by~\eqref{eulerfour}.}
\end{eqnarray*}
So if we write~$\phi(s)=\sum_{n\ge1}\frac{a_{S,n}}{n^s}$, then
\begin{eqnarray*}
\sum_{n\ge1}\frac{a_{S',n}}{n^s}&=&
\phi(s)\(\(\frac{q+1}{q-1}\)(1-q^{-s})\(\frac{1}{1-q^{1-s}}\)-\frac{2}{q-1}\)\\
&=&\phi(s)\(\frac{q^s+1}{q^s-q}\),
\end{eqnarray*}
showing the lemma for the set~$S'$.
All that remains is to check the case of a singleton~$S=\{p\}$, which is easy.
\end{proof}

\begin{example}\label{examplewithgeneralP}
If~$\orbit_S=s_P$ and~$\orbit_T=\zeta$ then,
by Lemmas~\ref{I think as I please} and~\ref{lemma:technicaltrick},
\[
\dirichlet_{S\times T}(s)=
\prod_{p\in P}\left(\textstyle\frac{1-p^{1-s}}{1+p^{-s}}\right)
\textstyle\frac{\zeta^2(s)\zeta(s-1)}{\zeta(2s)}
=\zeta(s)\displaystyle\prod_{p\notin P}\textstyle\frac{1+p^{-s}}{1-p^{1-s}}.
\]
\end{example}

\begin{example}\label{examplewithP=2}
Taking~$P=\{2\}$,~$\orbit_S=s_P$,~$\orbit_T=\zeta$ again, we have
\[
\dirichlet_{S\times T}(s)=
\(\textstyle\frac{1-2^{1-s}}{1+2^{-s}}\)
\textstyle\frac{\zeta^2(s)\zeta(s-1)}{\zeta(2s)}.
\]
The sequence~$\orbit_{S\times T}=(1, 1,
5, 1, 7, 5, 9, 1, 17,\dots)$ is~\seqnum{A035109},
which arises in work of Baake and
Moody~\cite[Eq.~(5.10)]{MR1636774}, where it is shown to
count the
elements of~$\mathbb Z^3$ with~$m$ distinct colours so that one
colour occupies a similarity sublattice of index~$m$ while the other
colours code the cosets.
\end{example}

\begin{example}
Let~$\dirichlet_S(s)=\zeta(s-a)$ and~$\dirichlet_T(s)=\zeta(s-b)$.
Then a calculation shows that
\[
\orbit_{S\times T}(n)=\frac{1}{n}\sum_{d\divides n}\mu(n/d)
\sigma_{a+1}(d)\sigma_{b+1}(d)
\]
(the details are in~\cite{apisit}), so Ramanujan's formula
gives
\[
\dirichlet_{S\times T}(s)=\frac{\zeta(s-a)\zeta(s-b)\zeta(s-a-b-1)}{\zeta(2s-a-b)}.
\]
\end{example}

These examples give an indication of how analytic properties
of~$\dirichlet_S$ and~$\dirichlet_T$ relate to those of~$\dirichlet_{T^k}$
and~$\dirichlet_{S\times T}$, and this is pursued in~\cite{apisit}.

\section{Counting in orbit monoids}

Counting in~$\monoid_T$ involves counting additive partitions,
with two changes: some parts may be missing
(that is, a ``restricted'' additive partition)
and some parts may come in several versions. Thus the
sequence~$\monoid_T$ is the
\emph{Euler transform} of the sequence~$\orbit_T$
(see Sloane and Plouffe~\cite[pp.20--22]{MR1327059}).

\begin{lemma}\label{lemma:eulertransform}
For any map~$T\in\mathfrak M$,
\begin{equation}\label{equation:gf}
1+\sum_{n=1}^{\infty}\monoid_T(n)s^n
=\prod_{i=1}^{\infty}\(1-s^i\)^{-\orbit_T(i)}
=\zeta_T(s)
\end{equation}
and
\begin{equation}\label{equation:recurrenceforGTn}
n\monoid_T(n)-\fix_T(n)-\sum_{k=1}^{n-1}\fix_T(k)\monoid_T(n-k)=0
\end{equation}
for all~$n\ge1$.
\end{lemma}

\begin{proof}
The first equality in~\eqref{equation:gf} is clear, since the
coefficient of~$s^n$ in the right-hand side counts partitions
of~$n$ into parts~$i$ with multiplicity~$\orbit_T(i)$; the
second equality is the usual Euler product expansion of the
dynamical zeta function. The recurrence
relation~\eqref{equation:recurrenceforGTn} may be seen by
expanding the zeta function as
\[
1+\sum_{n=1}^{\infty}\monoid_T(n)s^n=
\sum_{k=0}^{\infty}\frac{1}{k!}\(\sum_{n=1}^{\infty}\fix_T(n)
\frac{s^n}{n}\)^k
\]
and verifying that~\eqref{equation:recurrenceforGTn} satisfies
this relation.
\end{proof}

Thus the categories~$\mathfrak F$,~$\mathfrak O$ and~$\mathfrak
G$ are related as follows,
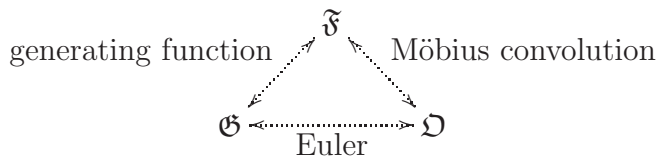
\[\xymatrix{&\mathfrak F\ar@{<.>}[ld]_{\mbox{generating function}}
\ar@{<.>}[rd]^{\mbox{M\"obius convolution\vphantom{g}}}
\\
\mathfrak G\ar@{<.>}[rr]_{\mbox{Euler}}&&\mathfrak O}
\]
and we indicate in this section how various growth properties
of any one sequence relate to growth properties of the others,
mostly by pointing out how these quantities arise in abstract
analytic number theory. These results extend those of Puri and
Ward~\cite{MR1873399} concerning relations between growth
in~$\fix_T$ and in~$\orbit_T$, and some related asymptotic
results are discussed in the paper of Baake and
Neum{\"a}rker~\cite{bn}. Before listing these, we discuss some
of the statements. It is often possible to estimate~$\fix_T(n)$
(or even to have a closed formula for~$\fix_T(n)$), and a
reasonable combinatorial replacement for ``hyperbolicity'' is
the assumption~\eqref{equation:hypothesisonFn} of a uniform
exponential growth rate in~$\fix_T$, where~$h$ plays the role
of topological entropy. A similar assumption often used in
abstract analytic number theory
is~\eqref{equation:hypothesisonGn} (hypotheses of this shape
are often called ``Axiom A'' or ``Axiom A$^\sharp$'' in number
theory). The assumption~\eqref{equation:hypothesisonGn} is
weaker than~\eqref{equation:hypothesisonFn}: it is pointed out
in~\cite{MR1100569} that there are arithmetic semigroups
with~$\monoid_T(n)e^{-hn}-C_{3}$ converging to zero
exponentially fast for which~$n\fix_T(n)e^{-hn}$ does not
converge. The hypothesis is weakened further
in~\eqref{equation:axiomA}, which is a permissive form of
exponential growth rate assumption. The
hypothesis~\eqref{equation:hypothesisonFn} fails for many
non-hyperbolic systems. If~$T$ is a quasihyperbolic toral
automorphism or a non-expansive~$S$-integer map with~$S$ finite
(see~\cite{MR1461206} or Example~\ref{example[7]})
then~\eqref{equation:hypothesisonFn} fails since the
ratio~$\fix_T(n+1)/\fix_T(n)$ does not converge
as~$n\to\infty$. In both cases the conclusions of
Theorem~\ref{theorem:mertensfromgrowthinFn}[1] also fail (see
Noorani~\cite{MR1787646} and Waddington~\cite{MR1139101} for
the case of a quasihyperbolic toral automorphism
and~\cite{emsw} for the case of~$S$-integer systems with~$S$
finite). As pointed out by Lindqvist and
Peetre~\cite{MR1609494}, Meissel considered the
sum~$\sum_{p}\frac{1}{p(\log p)^a}$ in~$1866$, and the
dynamical analogue of Meissel's theorem is given in
Theorem~\ref{theorem:mertensfromgrowthinFn}[4] below. In
Theorem~\ref{theorem:mertensfromgrowthinFn},~[1] is proved here
and~[2]--[4] are simply interpretations for orbit--counting of
well-known results in number theory.

\begin{theorem}\label{theorem:mertensfromgrowthinFn}
Let~$T$ be a map in~$\mathfrak M$.

\noindent{\rm[1]} Assume that there are\mc{growthinFnconstant1}
constants~$C_1>0,h>0$ and~$h'<h$ with
\begin{equation}\label{equation:hypothesisonFn}
\fix_T(n)=C_{1}e^{hn}+\bigo(e^{h'n}).
\end{equation}
Then
\begin{equation}\label{equation:growthinFnabstractprimenumbertheorem}
\orbit_T(n)=\frac{C_1}{n}e^{hn}+
\bigo\(e^{h'n}/n\),
\end{equation}
\begin{equation}\label{equation:growthinFndynamicalprimenumbertheorem}
\pi_T(N)=\frac{e^{h(N+1)}}{e^h-1}+\bigo\(e^{hN}/N^{3/2}\),
\end{equation}
and\mc{constantinmertensotherone}
\begin{equation}\label{equation:growthinFnmertens}
\mertens_T(N)=\sum_{n=1}^{N}\frac{\orbit_T(n)}{e^{hn}}=C_{1}
\sum_{n=1}^{N}\frac{1}{n}+C_{2}
+\bigo\(e^{h''N}\)
\end{equation}
where~$h''=\max\{h'-h,-h/2\}$, for some constant~$C_{2}$.

\noindent{\rm[2]} Assume that there are\mc{knpfconstantsa}
constants~$C_{3}>0,h>0$ and~$h'<h$ with
\begin{equation}\label{equation:hypothesisonGn}
\monoid_T(n)=C_{3}e^{hn}+\bigo(e^{h'n}).
\end{equation}
Then, for any~$\alpha>1$,\mc{knpfconstantsb}
\begin{equation*}\label{equation:knopfchpat8pi}
\pi_T(N)=C_{4}\frac{e^{hN}}{N}+\bigo\(e^{hN}/N^{\alpha}\)
\end{equation*}
and (equivalently)
\[
\fix_T(N)=e^{hN}+\bigo\(e^{hN}/N^{\alpha-1}\).
\]

\noindent{\rm[3]} Assume that there are\mc{axiomAconstant1}
constants~$C_{5}>0$,~$h>0$ with
\begin{equation}\label{equation:axiomA}
{\monoid_T(n)}=\(C_{5}+r(n)\){e^{hn}},
\end{equation}
where~$\displaystyle\sum_{n=0}^{\infty}\sup_{k\ge n}\vert
r(k)\vert<\infty$.
Then
\begin{equation}\label{equation:rhypothesisfirst}
\sum_{n=1}^{N}\frac{n\orbit_T(n)}{e^{hn}}=N+\bigo(1);
\end{equation}
\mc{mertensterm}
\begin{equation}\label{equation:rhypothesismertens}
\sum_{n=1}^{N}\frac{\orbit_T(n)}{e^{hn}}=
\log N+C_{6}+\bigo(1/N);
\end{equation}
\mc{multmertensterm}
\begin{equation*}\label{equation:rhypothesissecond}
\prod_{n=1}^{N}\(1-e^{-hn}\)^{\orbit_T(n)}=
\frac{C_{7}}{N}+\bigo(1/N^2);
\end{equation*}
and if in addition~$\zeta_T(-e^h)\neq0$ then,
for
any~$\lambda>1$,
\begin{equation}\label{equation:rhypothesisthird}
\orbit_T(n)=\frac{e^{hn}}{n}+\bigo\(e^{hn}/n^{\lambda}\)
\end{equation}
as~$n\to\infty$.

\noindent{\rm[4]} Assume that there are\mc{axiomAconstant2}
constants~$C_{8}>0$,~$h>0$ with
\begin{equation}\label{equation:meisselhypothesis}
\frac{\monoid_T(n)}{e^{hn}}=C_{8}+
\bigo\(1/\log(n)^{2+\epsilon}\)\mbox{as }n\to\infty.
\end{equation}
Then\mc{meisselconstant}
\begin{equation*}\label{equation:meissel}
\sum_{k=1}^{\infty}\frac{\orbit_T(k)}{e^{hk}k^{a}}
=\frac{1}{a}+C_{9}+\bigo(a)
\end{equation*}
as~$a\to0$.
\end{theorem}

\begin{proof}\noindent[1] The
estimate~\eqref{equation:growthinFnabstractprimenumbertheorem}
is easy to see; it is implicit in~\cite{MR1873399}
and~\cite{MR2085157} for example.
By~\eqref{fixintermsoforbits}, we have
\[
\fix_T(n)\ge n\orbit_T(n)\ge\fix_T(n)-\sum_{d\divides n,d<n}\fix_T(d),
\]
so
\[
C_1e^{hn}+\bigo(e^{h'n})\ge n\orbit_T(n)\ge C_1e^{hn}-n\(C_1e^{hn/2}+
\bigo(e^{h'n/2})\)
\]
which
gives~\eqref{equation:growthinFnabstractprimenumbertheorem}.
The proofs
of~\eqref{equation:growthinFndynamicalprimenumbertheorem} -- a
dynamical prime number theorem --
and~\eqref{equation:growthinFnmertens} -- a dynamical Mertens'
theorem -- use similar arguments to those in~\cite{MR2339472}
where a more delicate non-hyperbolic problem is studied.
Turning
to~\eqref{equation:growthinFndynamicalprimenumbertheorem},
notice
that~\eqref{equation:growthinFnabstractprimenumbertheorem}
implies that
\[
\left\vert\pi_T(N)-\sum_{n=1}^{N}\frac{C_1}{n}e^{hn}\right\vert
=\left\vert\sum_{n=1}^{N}\bigo\(e^{h'n}/n\)\right\vert
=\bigo\(e^{h'N}\).
\]
Now
\[
\left\vert\sum_{n=1}^{N}\frac{C_1}{n}e^{hn}
-\sum_{n=N-k(N)}^{N}\frac{C_1}{n}e^{hn}\right\vert\le\sum_{n=1}^{N-k(N)-1}C_1e^{hn}
=\bigo\(e^{h(N-k(N))}\)
\]
where~$k(N)=\lfloor N^{1/4}\rfloor$. Thus
\begin{eqnarray*}
\sum_{n=N-k(N)}^{N}\frac{C_1}{n}e^{hn}
&=&\frac{C_1e^{hN}}{N}\sum_{r=0}^{k(N)}e^{-hr}\(1-\textstyle\frac{r}{N}\)^{-1}\\
&=&\frac{C_1e^{hN}}{N}\left[\frac{e^{h}}{e^{h}-1}-\bigo\(e^{-hk(N)}\)
+\bigo\(\sum_{r=0}^{k(N)}\textstyle\frac{r}{N}\)\right]\\
&=&\frac{C_1e^{h(N+1)}}{e^h-1}+\bigo\(\frac{e^{hN}}{N^2}\sum_{r=0}^{k(N)}r\)\\
&=&\frac{C_1e^{h(N+1)}}{e^h-1}+\bigo\(e^{hN}/N^{3/2}\).
\end{eqnarray*}
Finally, notice that
\begin{equation}\label{equation:boundforproofofmertensunderFhypothesis}
\frac{\fix_T(n)}{ne^{hn}}-\frac{C_1}{n}=\frac{1}{n}\bigo\(e^{(h-h')n}\),
\end{equation}
so
\begin{equation}\label{equation:splitintotwotermsforMertensunderF}
\sum_{n=1}^{N}\frac{\orbit_T(n)}{e^{hn}}-C_1\sum_{n=1}^{N}\frac{1}{n}
=\sum_{n=1}^{N}\frac{1}{n}\(\frac{\fix_T(n)}{e^{hn}}-C_1\)
+\sum\frac{1}{n}\sum_{d\divides n,d<n}\mu(n/d)\frac{\fix_T(d)}{e^{hn}}.
\end{equation}
The bound~\eqref{equation:boundforproofofmertensunderFhypothesis}
shows that the two terms on the right-hand side
of~\eqref{equation:splitintotwotermsforMertensunderF} converge,
giving~\eqref{equation:growthinFnmertens} without error term. To see
the error term, notice that
\[
\left\vert \sum_{n=N+1}^{\infty}\frac{1}{n}\(\frac{\fix_T(n)}{e^{hn}}-C_1 \)
\right\vert
\le\sum_{n=N+1}^{\infty}\frac{1}{n}\bigo\(e^{(h'-h)n}\)
=\bigo\(e^{(h'-h)n}\)
\]
and
\[
\left\vert\sum_{n=N+1}^{\infty}\frac{1}{n}\sum_{d\divides n,d<n}
\mu(n/d)\frac{\fix_T(d)}{e^{hn}}\right\vert\le\sum_{n=N+1}^{\infty}
\(\frac{e^{hn/2}}{e^{hn}}+\bigo\(e^{-hn/2}\)\)=\bigo\(e^{-hN/2}\).
\]

\noindent[2] These are standard results from
Knopfmacher~\cite[Ch.~8]{MR545904}.

\noindent[3] The results~\eqref{equation:rhypothesisfirst}--\eqref{equation:rhypothesisthird}
are shown in~\cite{MR545904} to be consequences of Knopfmacher's
Axiom~$A^{\#}$ in~\eqref{equation:axiomA};~\eqref{equation:rhypothesisthird} is due
to Indlekofer~\cite{MR1237001}.

\noindent[4] This is shown by Wehmeier~\cite{w}.
\end{proof}

Standard estimates for the harmonic series
allow~\eqref{equation:growthinFnmertens} to be simplified; for
example under the hypothesis of
Theorem~\ref{theorem:mertensfromgrowthinFn}[1] we have
\[
\mertens_T(N)=C_1\log N+\bigo(1/N).
\]
Notice that the statements~\eqref{equation:growthinFnmertens}
and~\eqref{equation:rhypothesismertens} are versions of what is
usually called a dynamical Mertens' theorem, though they may
equally be seen in more elementary terms as consequences
of~$\frac{\orbit_T(n)}{e^{hn}}$ being close to~$\frac{1}{n}$
and the Euler--Maclaurin summation formula.
Theorem~\ref{theorem:mertensfromgrowthinFn}[1] has the
following kind of consequence: If~$T$ is a hyperbolic toral
automorphism or mixing shift of finite type with entropy~$h$,
then\mc{mertenstermlater}
\[
\sum_{\vert\tau\vert\le n}\frac{1}{e^{h\vert\tau\vert}}=\log n+C_{10}+
\bigo(1/n)
\]
and\mc{multmertenstermlater}
\[
\prod_{\tau}\(1-e^{-h\vert\tau\vert}\)^{-1}=\frac{C_{11}}{n}+\bigo(1/n^2)
\]
where~$\tau$ runs over the closed orbits of~$T$.
A stronger hypothesis
than~\eqref{equation:meisselhypothesis},\mc{constantforzetapole}
\begin{equation}\label{zetapolehypothesis}
\zeta_T(z)\sim\frac{C_{12}}{1-e^hz}
\end{equation}
as~$z\to e^{-h}$ with~$0<z<e^{-h}$, is considered
in~\cite{MR2263522}, where it is shown to
give\mc{constantforzetapoleB}
\[
\sum_{n=1}^{N}\frac{\orbit_T(n)}{e^{hn}}=\sum_{n=1}^{N}\frac{1}{n}+
C_{13}+\littleo(1);
\]
hence
\[
\prod_{n=1}^{N}\(1-e^{hn}\)^{-\orbit_T(n)}=C_{12}
e^{\gamma}N+\littleo(N)
\]
and
\begin{equation}\label{equation:littleoNbound}
\sum_{n=1}^{N}\frac{\fix_T(n)}{e^{hn}}=N+\littleo(N).
\end{equation}
As pointed out in~\cite{MR2263522}, the
hypothesis~\eqref{zetapolehypothesis} does not permit any
smaller error in~\eqref{equation:littleoNbound} for the
following reason. For any sequence~$(w_n)$ of non-negative
integers with~$\sum_{n=1}^{\infty}\frac{w_n}{n}<\infty$,
choose~$(a_n)$with~$0\le a_n<n$ and~$a_n\equiv
2^n+w_n2^n\pmod{n}$ and consider a map~$T$
with
\[
\orbit_T(n)=1+(2^n+w_n2^n-a_n)/n.
\]
This gives property~\eqref{zetapolehypothesis}, and a
calculation shows that the error term
in~\eqref{equation:littleoNbound} is as big
as~$\bigo\(\sum_{n=1}^{N}w_n\)$.

For a map~$T\in\mathfrak M$ with infinitely many orbits, it is
clear that~$\monoid_T$ is isomorphic to~$\sum_{\mathbb
N}\mathbb N_0$ as a semigroup. The information about how many
orbits~$T$ has of each length is contained in the weight
function~$\partial$, and in each example we compute the size of
the level set~$\monoid_T(n)$ for each~$n\ge1$. If the
sequence~$\fix_T$ is a linear recurrence sequence, then the
relation~\eqref{equation:recurrenceforGTn} shows
that~$\monoid_T$ is also a linear recurrence sequence.
Example~\ref{example[4]} shows that~$\monoid_T$ may satisfy a
relation of smaller degree, while Example~\ref{example[3]}
shows that~$\monoid_T$ may be of higher degree.

\begin{example}\label{example[1]}
Let~$T:X\to X$ be the golden mean shift, so
that~$\fix_T=(1,3,4,7,\dots)$ is the Lucas
sequence~\seqnum{A000032}.
By~\eqref{orbitsintermsoffix},~$\orbit_T$ is~\seqnum{A006206}.
Write~$\tau_i$ for the unique orbit of length~$i$ for~$1\le
i\le 4$, and write~$\tau_5^{(1)}$,~$\tau_5^{(2)}$ for the two
orbits of length~$5$. Then the elements in~$\monoid_T$ with
weight~$5$
are
\[
\tau_5^{(1)},\tau_5^{(2)},\tau_4+\tau_1,\tau_3+\tau_2,\tau_3+2\tau_1,\tau_2+3\tau_1,2\tau_2+\tau_1,5\tau_1,
\]
so~$\monoid_T(5)=8$.
The relation~\eqref{equation:gf} shows that~$\monoid_T(n)$ is
the~$(n+1)$st Fibonacci number, so~$\monoid_T$
is~\seqnum{A000045}.
\end{example}

\begin{example}\label{example[2]}
Let~$T\in\mathfrak M$ have~$\dirichlet_T(s)=\zeta(s)$. Then
there is a one-to-one correspondence between elements
of~$\monoid_T$ and partitions of natural numbers,
so~$\monoid_T$ is the classical partition
function~\seqnum{A000041}.
\end{example}

\begin{example}\label{example[3]}
Let~$T:X\to X$ be the full shift on~$a$ symbols, so
that~$\zeta_T(s)=\frac{1}{1-as}$;~\eqref{equation:gf} shows
that~$\monoid_T(n)=a^n-a^{n-1}$ for all~$n\ge1$. Thus~$\fix_T$
in this case is a linear recurrence of degree~$1$
while~$\monoid_T$ is a linear recurrence of degree~$2$.
\end{example}

\begin{example}\label{example[4]}
Let~$X=\mathbb Z[\frac{1}{6}]$ and let~$T:X\to X$ be the map
dual to~$r\mapsto\frac{2}{3}r$ on~$\mathbb Z[\frac16]$.
Then~$\fix_T(n)=3^n-2^n$ is~\seqnum{A001047} (a linear
recurrence of degree~$2$) by~\cite{MR961739}
so~$\monoid_T(n)=3^{n-1}$ by~\eqref{equation:gf},
and~$\monoid_T$ is~\seqnum{A000244} (a linear recurrence of
degree~$1$). More generally, if~$b>a>0$ are coprime integers,
then the dual map to~$r\mapsto\frac{a}{b}r$ on~$\mathbb
Z[\frac{a}{b}]$ has
\[
\zeta_T(s)=\frac{1-as}{1-bs},
\]
so~$\monoid_T(n)=(b^n-ab^{n-1})$ for all~$n\ge1$.
\end{example}

\begin{example}\label{example[5]}
The quadratic map~$T$ from Example~\ref{feigenbaumquadraticmap}
has a particularly simple monoid:~$\monoid_T$ is the binary
partition function~\seqnum{A018819} (the number of partitions
of~$n$ into powers of~$2$; by Sloane and
Sellers~\cite{MR2137567} this is also the number of
``non-squashing'' partitions of~$n$).
\end{example}

\begin{example}\label{example[6]}
An example similar in growth rate to Example~\ref{example[5]}
is studied in~\cite{emsw}: the map dual to~$x\mapsto 2x$ on the
localization~$\mathbb Z_{(3)}$ at the prime~$3$
has~$\fix_T(n)=\vert2^n-1\vert_3^{-1}$, so
\[
\orbit_T(n)=\begin{cases} 1&\mbox{if }n=1\mbox{ or }2\cdot 3^k,
k\ge1;\\0&\mbox{if not,}\end{cases}
\]
and therefore~$\monoid_T(2n+1)$ is the number of partitions
of~$6n+3$ into powers of~$3$ for~$n\ge0$,
and~$\monoid_T(2n)=\monoid_T(2n+1)$ for~$n\ge1$.
\end{example}

\begin{example}\label{example[7]}
An example of a map that is not hyperbolic but still has
exponentially many periodic orbits is given by the simplest
non-trivial $S$-integer map dual to~$x\mapsto2x$ on~$\mathbb
Z[1/3]$. By~\cite{MR1461206} this
has~$\fix_T(n)=(2^n-1)\vert2^n-1\vert_3$, and a calculation
shows that~$\orbit_T$ is~\seqnum{A060480}, and
thus~$\monoid_T=(1,1,3,4,10,13,33,56,10,\dots)$.\end{example}

%\bibliographystyle{uwab}

%\bibliography{refs}

\providecommand{\bysame}{\leavevmode\hbox to3em{\hrulefill}\thinspace}

\bigskip
\hrule
\bigskip

\noindent 2000 {\it Mathematics Subject Classification}:
Primary 11B83.

\noindent \emph{Keywords: }
Closed orbits; iteration; functorial properties; Dirichlet series

\bigskip
\hrule
\bigskip

\noindent Concerned with sequences:\\\seqnum{A000032},
\seqnum{A000041}, \seqnum{A000045}, \seqnum{A000244},
\seqnum{A001047}, \seqnum{A006206}, \seqnum{A018819},
\seqnum{A027377}, \seqnum{A027381}, \seqnum{A035109},
\seqnum{A036987}, \seqnum{A038712}, \seqnum{A060480},
\seqnum{A060648}, \seqnum{A065333}, \seqnum{A091574}.

\bigskip
\hrule
\bigskip

\end{document}